\documentclass[reqno, 11pt]{article}

\usepackage{enumerate}
\usepackage{needspace}
\usepackage{latexsym}
\usepackage{amsmath}
\usepackage{amsfonts}
\usepackage{amsthm}
\usepackage{amssymb,mathtools}
\usepackage{newlfont}
\usepackage{graphics}
\usepackage{color}
\usepackage{float}
\usepackage{diagbox}
\usepackage{tocloft}
\usepackage{titlesec}
\usepackage{booktabs,array,longtable,enumitem,microtype}
\usepackage{extpfeil}
\usepackage{centernot}

\usepackage[linesnumbered,ruled,vlined]{algorithm2e}
\usepackage{url}

\usepackage[margin=1in]{geometry}

\usepackage[T1]{fontenc}
\usepackage{lmodern}
\usepackage{microtype}
\usepackage{rotating}
\usepackage{multirow}
\usepackage{extarrows}
\usepackage[sort,compress,numbers]{natbib}
\usepackage[utf8]{inputenc}
\usepackage{xcolor}
\usepackage{listings}
\usepackage{aliascnt}
\usepackage[pagebackref,colorlinks=true,linkcolor=blue,citecolor=red,urlcolor=blue]{hyperref}
\usepackage[nameinlink,noabbrev,capitalize]{cleveref}
\numberwithin{equation}{section}

\setlist[enumerate]{label=\textup{(\roman*)},leftmargin=*}
\newtheorem{theorem}{Theorem}[section]
\newaliascnt{proposition}{theorem}
\newtheorem{proposition}[proposition]{Proposition}
\aliascntresetthe{proposition}
\newaliascnt{lemma}{theorem}
\newtheorem{lemma}[lemma]{Lemma}
\aliascntresetthe{lemma}
\newaliascnt{corollary}{theorem}

\aliascntresetthe{corollary}

\theoremstyle{definition}
\newaliascnt{definition}{theorem}
\newtheorem{definition}[definition]{Definition}
\aliascntresetthe{definition}
\newaliascnt{remark}{theorem}
\newtheorem{remark}[remark]{Remark}
\aliascntresetthe{remark}
\newaliascnt{example}{theorem}
\newtheorem{example}[example]{Example}
\aliascntresetthe{example}
\newaliascnt{conjecture}{theorem}
\newtheorem{conjecture}[conjecture]{Conjecture}
\aliascntresetthe{conjecture}
\newaliascnt{question}{theorem}

\aliascntresetthe{question}
\newaliascnt{problem}{theorem}

\aliascntresetthe{problem}
\newaliascnt{alg}{theorem}

\aliascntresetthe{alg}

\allowdisplaybreaks[4]

\SetKwInput{KwInput}{Input}                % Set the Input
\SetKwInput{KwOutput}{Output}              % set the Output

\newcommand{\Z}{\mathbb Z}
\newcommand{\R}{\mathbb R}

\newcommand{\ehr}{\operatorname{ehr}}

\newcommand{\Id}{I}
\newcommand{\EE}{\mathcal E}
\newcommand{\Aop}{\mathcal A}
\newcommand{\BB}{\mathcal B}
\newcommand{\DD}{\mathcal D}
\newcommand{\NN}{\mathcal N}

\title{Ehrhart positivity for paving positroids}

\author{Feihu Liu$^{\color{blue} \dag}$ and Zihao Zhang$^{\color{blue} \S}$
\\[2mm]
{\small $^{\color{blue} \dag}$ Center for Combinatorics, LPMC,}\\[-0.8ex]
{\small Nankai University, Tianjin 300071, P.R.~China}\\
{\small $^{\color{blue} \S}$ School of Mathematics and Statistics,}\\[-0.8ex]
{\small Beijing Institute of Technology, Beijing 102400, P.R.~China}\\
{\small {\color{blue} $^\dag$} Email address: liufeihu7476@163.com}\\
{\small {\color{blue} $^\S$} Email address: zihao-zhang@foxmail.com}\\
}

\date{\today}

\begin{document}

\maketitle

\begin{abstract}
A central conjecture of Ferroni, Jochemko, and Schr\"{o}ter posits that every positroid is Ehrhart positive. While Ehrhart positivity has recently been verified for notable classes such as lattice path matroids and series-parallel matroids, the general conjecture remains open. In this paper, we prove that the base polytope of any paving positroid has positive Ehrhart coefficients. This establishes a new class of Ehrhart positive positroids and provides substantial evidence supporting this conjecture.
\end{abstract}

\noindent
\begin{small}
\end{small}

\noindent
\begin{small}
\end{small}

%\tableofcontents

\section{Introduction}

A \emph{lattice polytope} is the convex hull of finitely many points of $\mathbb{Z}^d$.
Throughout this paper, we assume that $d\geq 1$.
Let $P\subset\mathbb{R}^d$ be a full-dimensional lattice polytope, i.e., $\dim P=d$. 
For $t\in\mathbb{N}$, its \emph{lattice-point enumerator}
$$\ehr_P(t)=|tP\cap\Z^d|.$$
counts the number of integer points in the $t$-th dilation $tP = \{ t\alpha : \alpha \in P \}$ of $P$.

Ehrhart~\cite{Ehrhart1962} proved that the function $\ehr_P(t)$ is a polynomial in $t$ of degree $d$ with constant term $1$.
The resulting polynomial is called the \emph{Ehrhart polynomial} of $P$.
Let $\ehr_P(t)=c_dt^d+\cdots+ c_1t+ c_0$.
It is well known \cite[Corollary 3.20; Theorem 5.6]{BeckRobins2015} that $c_d$ equals the volume of $P$, $c_{d-1}$ equals half of the boundary volume of $P$, and the constant term is $c_0=1$. Therefore, these three coefficients are always positive.
The remaining coefficients, however, are more intricate and lack a straightforward geometric interpretation \cite{McMullen77}.

A lattice polytope $P$ is said to be \emph{Ehrhart positive} if $c_i>0$ for all $i$.
Ehrhart positivity has attracted considerable attention in recent years. 
Many families of polytopes have been shown to be Ehrhart positive, including lattice path matroids \cite{FMP2026} 
(which encompass the hypersimplices~\cite{Ferroni2021}, minimal matroids~\cite{Ferroni2022Minimal}, Schubert matroids~\cite{Fan-Li}, Catalan matroids~\cite{Chen-Li-Yao}), rank-two matroids~\cite{FerroniJochemkoSchroter2022}, cross-polytopes~\cite[Section~2]{Liu2019}, the $y$-families of generalized permutohedra~\cite{Postnikov2009} (which encompass the Pitman--Stanley polytopes~\cite{Stanley-Pitman}), and cyclic polytopes~\cite{LiuFuA-M}.
Recently, Li and Shao \cite{LiShao2026} showed that Series-parallel matroids are Ehrhart positive.
See \cite{Liu2019} for a survey of Ehrhart positivity.

Reeve's tetrahedra provide the classical non-Ehrhart positive examples \cite{Reeve1957}.
Other examples that are not Ehrhart positive include order polytopes~\cite{LiuTsuchiya2019,LiuXinZhang2026} and smooth polytopes~\cite{CastilloEtAl2018}, among others \cite{Hibi-Higashitani-Tsuchiya-Yoshida,LiuTaoXin,LiuTaoXinSign}.
Specifically, Ferroni \cite{Ferroni2022Counterexample} proved that matroid base polytopes are not necessarily Ehrhart-positive.
Indeed, he constructed sparse paving matroids with negative Ehrhart coefficients. Thus even the paving hypothesis alone does not imply positivity.

This paper focuses on the Ehrhart positivity of paving positroid. 
Jiang gave a combinatorial formula for the $h^*$-polynomials of connected positroid polytopes \cite[Theorem~1.1]{Jiang2025}.
We now present the relevant definitions.

A \emph{matroid} $M$ is a pair $(E, \mathcal{B})$ where $E$ is finite set and $\mathcal{B}$ is family of subsets of $E$, i.e. $\mathcal{B} \subseteq 2^E$ that satisfies the following two conditions: (i): $\mathcal{B} \neq \emptyset$. 
(ii): For each $B_1 \neq B_2$ members of $\mathcal{B}$ and $a \in B_1 \setminus B_2$, there exists an element $b \in B_2 \setminus B_1$ such that $(B_1 \setminus \{a\}) \cup \{b\} \in \mathcal{B}$.
We refer to Condition (ii) as the \emph{basis exchange property}.
The elements of $\mathcal{B}$ are called the \emph{basis} elements of $M$.
Note that all the bases of a matroid have the same cardinality.

Let $M = (E, \mathcal{B})$ be a matroid.
\begin{itemize}
\item A subset $I \subseteq E$ is called \emph{independent} if $I$ is contained in some \(B \in \mathcal{B}\). 
      When a subset of $E$ is not independent, we say it is \emph{dependent}.
    
\item A dependent subset $C \subseteq E$ is called a \emph{circuit} if every proper subset of $C$ is independent.
      In other words, a circuit is a minimal dependent set.

\item For any $A \subseteq E$, the \emph{rank} of $A$, denoted by $r_M(A)$, is defined as
    $$r_M(A) = \max_{B \in \mathcal{B}} |A \cap B|.$$
    In other words, the rank $r_M(A)$ of $A\subseteq E$ is the largest size of an independent subset of $A$.
    The \emph{rank} of $M$ is defined to be $r(M):=r_M(E)$. The \emph{corank} of $M$ is $|E|-r(M)$.
    The closure of $A$ is $\operatorname{cl}_M(A)=\{e\in E:r_M(A\cup\{e\})=r_M(A)\}$.
    
\item A subset $F \subseteq E$ is called a \emph{flat} if $r_M(F \cup \{e\}) > r_M(F)$ for every $e \notin F$.

\item A flat $H$ satisfying $r_M(H) = r_M(E) - 1$ is called a \emph{hyperplane}. A hyperplane is dependent if its size is at least $r(M)$.

\item An element $e\in E$ is a \emph{loop} if it belongs to no basis, and a \emph{coloop} if it belongs to every basis.
\end{itemize}

The \emph{direct sum} of matroids on disjoint sets is defined by taking as its bases the unions of one basis from each summand. 
A matroid is \emph{connected} if it is not a direct sum of two matroids on nonempty sets. 
The \emph{connected components} of a matroid are the summands in its unique decomposition into connected matroids.

The \emph{uniform matroid} $U_{r,n}$, where $0\leq r\leq n$, has ground set $[n]=\{1, 2, \ldots,n\}$ and all $r$-element subsets as bases.
A rank-$r$ matroid is \emph{paving} if each circuit has at least $r$ elements. 
In particular, every set of at most $r-1$ elements is independent. 
Postnikov \cite{Postnikov-positroid} introduced the definition of positroid in the context of studying totally nonnegative Grassmannian.
A real matrix $A \in \mathbb{R}^{r \times n}$ of rank $r$ is said to be \emph{totally nonnegative} if all of its $r \times r$ minors $\Delta_I(A)$ are nonnegative, where $I \in \binom{[n]}{r}$ is assumed to be ordered in the standard increasing fashion $i_1 < i_2 < \dots < i_r$. 
A \emph{positroid} of rank $r$ on the set $[n]$ is defined as the column matroid $M$ realized by a totally nonnegative matrix $A$. Equivalently, the basis $\mathcal{B}(M)$ of $M$ consists precisely of the $r$-element subsets of columns corresponding to the positive $r\times r$ minors of $A$.
In this context, the underlying linear order on $[n]$ is canonically viewed cyclically modulo $n$, reflecting the fundamental property that the class of positroids is invariant under cyclic shifts of the ground set, see \cite{ARW2016}.

Let $E$ be a finite set of cardinality $n$, which we canonically identify with $[n] = \{1, 2, \ldots, n\}$. 
For a vector $x \in \mathbb{R}^E$ and a subset $A \subseteq E$, we employ the standard shorthand $x(A)= \sum_{e \in A} x_e$. 
The indicator vector $\mathbf1_A\in\R^E$ has coordinate $1$ on $A$ and coordinate $0$ on $E\setminus A$.
Given a matroid $M$ of rank $r$ on the ground set $E$, its \emph{base polytope} $P(M)$ is defined as the convex hull of the indicator vectors of its bases:
$$P(M)= \operatorname{conv} \{ \mathbf{1}_B \mid B \text{ is a basis of } M \} \subset \mathbb{R}^n.$$
The \emph{hypersimplex} (also \emph{uniform matroid polytope}) $\Delta(r, n)$ is given by
$$\Delta(r, n) = \left\{ x \in [0, 1]^n \;\middle\vert\; x([n]) = r \right\}.$$
The hypersimplex is $P(U_{r,n})$.
Finally, we denote the Ehrhart polynomial of the base polytope $P(M)$ by $\ehr_M(t)$.

We use the rank description of a matroid base polytope:
\[P(M)=\{x\in\R_{\geq0}^{E}:x(E)=r(M),\ x(A)\leq r_M(A)\text{ for all }A\subseteq E\}.
\]
This is the standard base polytope theorem, see \cite[Proposition~5.4]{ARW2016} or \cite[Proposition 2.3]{Feichtner05}.
The standard dimension formula is $\dim P(M)=|E|-c(M)$, where $c(M)$ is the number of connected components
\cite[Section~2.2]{HanelyEtAl2023}. The face $x(A)=r_M(A)$ is the base polytope of $(M|A)\oplus(M/A)$
\cite[Proposition~5.3]{ARW2016}.

For $A\subseteq E$, the \emph{restriction} $M|A$ has ground set $A$ and the independent sets of $M$ contained in $A$. 
The \emph{deletion} is $M\setminus A=M|(E\setminus A)$. 
The \emph{contraction} $M/A$ has ground set $E\setminus A$ and rank function $r_{M/A}(X)=r_M(X\cup A)-r_M(A)$ for $X\subseteq E\setminus A$.
If $B_0$ is independent, the bases of $M/B_0$ are the sets $B\setminus B_0$ for bases $B$ of $M$ containing $B_0$.
A minor is a matroid obtained by a sequence of deletions and contractions.
We refer to \cite[Chapters~1--4]{Oxley2011} for these matroid operations and connectivity facts.

The main motivation for this paper comes from the following conjecture.

\begin{conjecture}(Ferroni--Jochemko--Schr\"oter; \cite[Conjecture~6.3]{FerroniJochemkoSchroter2022})
All positroids are Ehrhart positive.
\end{conjecture}

Ferroni, Jochemko, and Schr\"oter \cite[Theorem 1.2]{FerroniJochemkoSchroter2022} proved Ehrhart positivity for all positroids of rank two.
Recently, two important subclasses of positroids have been shown to be Ehrhart positive: 
lattice path matroids, by Ferroni, Morales, and Panova \cite[Theorem~1.2]{FMP2026}, and all series-parallel matroids, by Li and Shao \cite[Theorem~1.3]{LiShao2026}.

The main result of this paper is that the following subclass is Ehrhart positive.

\begin{theorem}\label{thm:main}
Every paving positroid is Ehrhart positive. More precisely, if $M$ is a paving positroid and $d=\dim P(M)$, then
$$\ehr_M(t)=\sum_{j=0}^{d}c_jt^j, \qquad c_j>0\quad\text{for }0\leq j\leq d.$$
\end{theorem}

Machacek and Nasr described coloopless paving positroids by Le-diagrams \cite[Theorem~5.8 and Lemma~5.11]{MachacekNasr2024v2}.
Their work on sparse paving positroids gives characterizations by Grassmann necklaces, decorated permutations, and Le-diagrams, together with an enumeration \cite[Theorems~3.2, 4.2, and~5.4]{MachacekNasr2026Sparse}.
The geometric description used here follows from the cyclic interval facet characterization of positroids.

The starting point is the established description of a paving base polytope as a hypersimplex with regions removed at its dependent
hyperplanes. Hanely et al. computed the corresponding Ehrhart polynomials \cite[Proposition~5.10 and Theorem~5.11]{HanelyEtAl2023}.
Deligeorgaki, McGinnis, and Vindas-Mel\'endez proved positivity for panhandle matroids \cite[Theorems~1.3]{DMV2026}.
For the Ehrhart polynomial of paving matroid, they \cite[Theorems~1.4, Section~5]{DMV2026} give a coefficient-wise upper bound by the hypersimplex. However, they do not give a positive lower bound.

The paper is organized as follows. 
\cref{sec:geometry} establishes the cyclic interval constraints
and disjoint forbidden regions, and introduces a running example.
\cref{sec:integrals} gives the root-summation identity and a
periodic integral representation of the binomial transform. Each
coordinate contributes a zero or unit step with an evaluation weight,
or an integrated step in $[0,1]$ with a factor of $t$.
\cref{sec:constraints} imposes each interval constraint on a
consecutive block of steps and computes its nonnegative boundary
correction. It proves polynomiality of the rooted counts and uses
disjointness in the common path representation to obtain coefficient
nonnegativity. \cref{sec:main-proof} recovers the Ehrhart polynomial
by root summation and uses connected coordinate faces to supply a
positive integral in every required degree. All integrals are Riemann
integrals over finite intervals or polytopes.

\section{The geometry of paving positroids}\label{sec:geometry}

A cyclic interval consists of consecutive elements in the cyclic order; it may cross the chosen first element. 
Ardila, Rinc\'on, and Williams states that a matroid in a specified order is a positroid if and only if its base polytope is described, in the hyperplane $x(E)=r(M)$, by upper bounds on cyclic interval sums
\cite[Proposition~5.6]{ARW2016}.

\begin{lemma}\label{lem:geometry}
Let $M$ be a connected rank-$r$ paving positroid on a cyclically ordered
set $E$ of size $n$, where $2\leq r\leq n-2$. Let
$H_1,\ldots,H_\ell$ be its hyperplanes of size at least $r$. Then
\[
 P(M)=\{x\in\Delta(r,n):x(H_i)\leq r-1
                                  \text{ for }1\leq i\leq\ell\}.
\]
Each $H_i$ is an interval in the same cyclic order, and
\[
 |H_i\cap H_j|\leq r-2\qquad(1\leq i<j\leq\ell).
\]
The regions $x(H_i)>r-1$ inside $\Delta(r,n)$ are pairwise disjoint.
\end{lemma}

\begin{proof}
If $r_M(A)\leq r-2$ and $A$ were dependent, a basis of $A$ together
with one additional element would contain a circuit of size at most
$r_M(A)+1<r$. Hence such an $A$ is independent, and its rank inequality
follows from $x_e\leq1$. A set of rank $r-1$ is contained in a
hyperplane. A hyperplane of size $r-1$ gives another consequence of
the coordinate bounds, whereas the remaining hyperplanes are precisely
$H_1,\ldots,H_\ell$. The inequality for a spanning set follows from
nonnegativity and $x(E)=r$. The rank description therefore gives the
displayed description of $P(M)$.

Every $(r-1)$-set is independent and its closure is its unique
containing hyperplane. Distinct hyperplanes therefore intersect
in at most $r-2$ elements.

For each $H_i$, the restriction is $U_{r-1,|H_i|}$ and the contraction
is $U_{1,n-|H_i|}$. Connectedness implies $n-|H_i|\geq2$: a singleton
complement would be a coloop. The face $x(H_i)=r-1$ is the product of
the two uniform base polytopes and has dimension
$(|H_i|-1)+(n-|H_i|-1)=n-2$.
It is therefore a facet. The cyclic interval description of $P(M)$
must include an inequality supporting this facet. Its outward normal,
modulo the normal to $x(E)=r$, gives
\[
 \mathbf1_{H_i}=a\mathbf1_J+b\mathbf1_E
 \qquad\text{with }a>0
\]
for a proper nonempty cyclic interval $J$. The two values on the
right are $b$ and $a+b$; they must be $0$ and $1$ in that order.
Thus $a=1$, $b=0$, and $H_i=J$.

Finally, for distinct $H_i,H_j$ and $x\in\Delta(r,n)$,
\[
 x(H_i)+x(H_j)
 =x(H_i\cup H_j)+x(H_i\cap H_j)
 \leq r+|H_i\cap H_j|\leq2r-2.
\]
The two sums cannot both exceed $r-1$. 
The same disjoint cap geometry appears in the paving Ehrhart formulas of
\cite[Proposition~5.10]{HanelyEtAl2023}.
\end{proof}

For a fixed matroid $M$, an element $e$ is distinguished as the root.
The remaining elements are enumerated as $1,\ldots,q$ in cyclic order
after $e$, so $n=q+1$. Here $q$ denotes the number of nonroot elements;
it equals $\dim P(M)$ when $M$ is connected. For integers $t\geq1$
and $0\leq a\leq t$, the rooted count is
$F_e(a,t)=|\{x\in tP(M)\cap\Z^{q+1}:x_e=a\}|$.
When the matroid needs to be specified, the same count is denoted by
$F_{M,e}(a,t)$.

If $e\notin H_i$, then $H_i$ is a consecutive block in $1,\ldots,q$,
and its constraint is $x(H_i)\leq(r-1)t$. If $e\in H_i$, its complement
is a consecutive block in $1,\ldots,q$, and its constraint is
$x(E\setminus H_i)\geq t$. Polynomiality of this rooted count on the
whole integer triangle will be proved in \cref{lem:rooted}.

\begin{example}\label{ex:four-element}
Let $M$ be the column matroid of
\[
 \begin{pmatrix}
  1&1&1&0\\
  0&0&1&1
 \end{pmatrix}.
\]
The ordered $2\times2$ minors are all $1$ except the zero minor on
$\{1,2\}$. Thus $M$ is a loopless rank-two positroid, hence paving,
with bases all two-element subsets except $\{1,2\}$. The circuits
$\{1,2\}$ and $\{1,3,4\}$ force every element into the same
direct summand, so $M$ is connected. Its only dependent hyperplane is
$H=\{1,2\}$, so \cref{lem:geometry} gives
\[
 P(M)=\{x\in\Delta(2,4):x_1+x_2\leq1\}.
\]

Take $e=4$, with nonroot coordinates $1,2,3$ in this order.
For integers $t\geq1$ and $0\leq a\leq t$, set $s=x_1+x_2$.
The equations $x_4=a$ and $x_1+x_2+x_3+x_4=2t$ give
$x_3=2t-a-s$. The bound $0\leq x_3\leq t$ and the hyperplane
constraint give $t-a\leq s\leq t$.
For each such $s$, there are $s+1$ choices for $(x_1,x_2)$, and
their upper bounds follow from $s\leq t$. Hence
\[
 F_e(a,t)=\sum_{s=t-a}^{t}(s+1)
        =(a+1)(t+1)-\frac{a(a+1)}2.
\]
\end{example}

\section{Binomial transforms and periodic integral representations}\label{sec:integrals}

\subsection{The binomial transform}

In a lattice-point count, $t$ is a nonnegative integer. After a
polynomial extension has been established, $t$ is treated as an
indeterminate.
For a formal power series $A(z)$, the notation $[z^j]A(z)$ denotes the
coefficient of $z^j$, with $[z^j]A(z)=0$ for $j<0$. The same convention
is used for other series variables, including $v$ and $\lambda$.

Binomial coefficients are interpreted as polynomials in their upper
argument: $\binom{u}{j}=u(u-1)\cdots(u-j+1)/j!$ for integers
$j\geq1$, and $\binom{u}{0}=1$.

For a polynomial $F(a,t)$ of degree at most $d$ in $a$, there is a
unique expansion
\[
 F(a,t)=\sum_{s=0}^{d}h_s(t)\binom{a+s}{s}.
\]
Its \emph{binomial transform} is the polynomial
$G(t,X)=(\NN F)(t,X)=\sum_{s=0}^{d}h_s(t)X^s$.
When $F_e$ has a polynomial extension, we denote its transform by
$G_e(t,X)=(\NN F_e)(t,X)$.
If $F$ has total degree at most $d$, then
$\deg_t h_s\leq d-s$. This follows by successively comparing the
coefficients of $a^d,a^{d-1},\ldots,1$ in the expansion.
For an integer $t\geq0$, summation over the root coordinate gives
\[
 \sum_{a=0}^{t}F(a,t)
     =\sum_{s=0}^{d}h_s(t)\binom{t+s+1}{s+1}.
\]
The identity follows from
$\binom{a+s}{s}=\binom{a+s+1}{s+1}-\binom{a+s}{s+1}$
and telescoping over $0\leq a\leq t$.

The functional that will recover the Ehrhart polynomial is
\[
 \Gamma_v(G)=\sum_{s=0}^{d}h_s(t)[\lambda^s]
       \frac{(1-\lambda(1-v))^{t-1}}{(1-\lambda)^{t+1}},
 \qquad 0\leq v\leq1.
\]
The powers are expanded at $\lambda=0$. Only finitely many
coefficients in $\lambda$ occur, and each is a polynomial in $t$ and $v$.

\begin{lemma}\label{lem:root-summation}
For $G(t,X)=\sum_{s=0}^{d}h_s(t)X^s$, one has
\[
 \sum_{s=0}^{d}h_s(t)\binom{t+s+1}{s+1}
       =G(t,1)+t\int_0^1\Gamma_v(G)\,dv.
\]
\end{lemma}

\begin{proof}
Differentiation with respect to $v$ gives the formal identity
\[
 \frac{\partial}{\partial v}(1-\lambda(1-v))^t
       =t\lambda(1-\lambda(1-v))^{t-1}.
\]
Integrating coefficientwise over $0\leq v\leq1$ yields
\[
 \frac1{1-\lambda}
 +t\int_0^1
       \frac{(1-\lambda(1-v))^{t-1}}{(1-\lambda)^{t+1}}\,dv
 =\frac{(1-\lambda)^{-t-1}-1}{\lambda}.
\]
The numerator on the right has zero constant term, so the quotient
is a formal power series. Its coefficient of $\lambda^s$ is
$\binom{t+s+1}{s+1}$. On the left this coefficient is
$1+t\int_0^1\Gamma_v(X^s)\,dv$.
Multiplication by $h_s(t)$ and summation prove the formula.
All extracted coefficients are polynomials, so the identity also
includes $t=0$.
\end{proof}

\cref{lem:rooted} will establish polynomiality of the rooted
count for the paving positroids under consideration. The root sum
will then take the form
\begin{equation}\label{eq:root-goal}
 \ehr_M(t)=G_e(t,1)+t\int_0^1\Gamma_v(G_e)\,dv.
\end{equation}
Thus coefficientwise nonnegativity of the local term $G_e(t,1)$
and of $\Gamma_v(G_e)$ will imply Ehrhart nonnegativity.
To prove these assertions while imposing the hyperplane constraints,
we will allow the integral data and the two endpoint values of the
functional to vary independently.

\begin{example}[\cref{ex:four-element}, continued]
The rooted polynomial has the binomial expansion
\[
 F_e(a,t)=(t+2)\binom{a+1}{1}-\binom{a+2}{2}.
\]
Its transform is therefore
\[
 G_e(t,X)=(t+2)X-X^2.
\]
The negative coefficient of $X^2$ shows why positivity cannot be
read directly from the coefficients of the binomial transform.
\end{example}

To construct the integral representation, first consider uniform
matroids. For integers $p\geq1$ and $1\leq k\leq p$, the rooted count of
$U_{k,p+1}$ is the number of $p$-tuples of integers in $[0,t]$ with sum $kt-a$.
Inclusion-exclusion gives the polynomial
\[
 F_{p,k}(a,t)=\sum_{j=0}^{k-1}(-1)^j\binom pj
                     \binom{(k-j)t-a+p-j-1}{p-1}.
\]
The displayed formula agrees with
the count for every integer $t\geq1$ and
$0\leq a\leq t$. Indeed, after $j$ specified coordinates exceeding
$t$ have been reduced by $t+1$, their remaining sum is
$N=(k-j)t-a-j$. Terms with $j\geq k$ have $N<0$ and are omitted.
For $0\leq j\leq k-1$, a nonnegative $N$ gives
$\binom{N+p-1}{p-1}$ by the elementary count of nonnegative integer
solutions. If $N<0$, then
$N\geq(k-j-1)t-j\geq-j\geq-(p-1)$. Hence
$0\leq N+p-1<p-1$, and the same binomial polynomial is zero.

For an integer $b\geq1$, the polynomial
\[
 f_b(T,X)=\sum_{s=0}^{b-1}(-1)^s
                         \binom{T+b}{b-1-s}X^s
\]
is the transform of $\binom{T-a+b-1}{b-1}$. Indeed,
\[
 \binom{T-a+b-1}{b-1}
 =[v^{b-1}](1+v)^{T+b}(1+v)^{-a-1}
 =\sum_{s=0}^{b-1}(-1)^s\binom{T+b}{b-1-s}\binom{a+s}{s}.
\]
The equivalent coefficient formulas are
\[
 f_b(T,X)
 =[v^{b-1}]\frac{(1+v)^{T+b}}{1+vX}
 =[\lambda^{b-1}]
       \frac{(1-\lambda)^{-T-1}}{1-\lambda(1-X)}.
\]
The second follows from $v=\lambda/(1-\lambda)$ and the formal
coefficient change
\[
 [v^d]g(v)=[\lambda^d](1-\lambda)^{d-1}
                                  g\!\left(\frac\lambda{1-\lambda}\right).
\]
To verify this coefficient change, it suffices to consider $g(v)=v^k$
for integers $k\geq0$. The right-hand side is
$[\lambda^{d-k}](1-\lambda)^{d-1-k}$. It is zero when $k>d$ by the
negative-degree convention, and it is $1$ when $k=d$. When $k<d$,
the polynomial $(1-\lambda)^{d-1-k}$ has degree less than $d-k$, so
the coefficient is zero. The two sides therefore agree on every
monomial. Only the first $d+1$ coefficients of a formal series can
contribute, which proves the identity for arbitrary $g$.

Linearity of $\NN$ now gives the transform of the uniform rooted count:
\[
 G_{p,k}(t,X)=\sum_{j=0}^{k-1}(-1)^j\binom pj
                                f_p((k-j)t-j,X).
\]
\begin{example}
For $M=U_{1,3}$, the two nonroot coordinates have sum $t-a$.
There are $t-a+1$ such pairs for $0\leq a\leq t$, and each coordinate
lies in $[0,t]$. Thus
$F_e(a,t)=t-a+1=(t+2)\binom{a}{0}-\binom{a+1}{1}$, and
$G_e(t,X)=t+2-X$.
\end{example}

\subsection{An ordinary integral kernel}

Let $B>0$ and let $\varphi:[0,B]\longrightarrow[0,1]$ be affine
on the open intervals of a finite partition, with independently assigned
values at the partition points. For bounded piecewise continuous $g$, set
$(Dg)(u)=\varphi(u)g(u)+t\int_u^B g(v)\,dv$.
All integrals are ordinary Riemann integrals or finite multiple Riemann
integrals. Compositions retain evaluation terms separately from integral
terms; the coefficient of $g(v)$ in the latter is the \emph{integral kernel}.
For every integer $j\geq1$,
\[
 (D^jg)(u)=\varphi(u)^jg(u)
             +\int_u^B K_{D^j}(u,v)g(v)\,dv.
\]
Composition defines the kernel pointwise. Each integration contributes
$t$, so $K_{D^{s+1}}$ is divisible by $t$ for $s\geq0$.
For example, $K_{D^2}(u,v)/t=\varphi(u)+\varphi(v)+t(v-u)$ for $u<v$:
the first two terms have one integration, and the last has two.

\begin{lemma}\label{lem:resolvent}
For $0\leq u<v\leq B$ and an integer $s\geq0$, the integral kernel
of $D^{s+1}$ satisfies
\[
 \frac{K_{D^{s+1}}(u,v)}t=[\lambda^s]R_\lambda(u,v),
 \qquad
 R_\lambda(u,v)=
 \frac{\exp\!\left(t\lambda\int_u^v
                         \frac{dw}{1-\lambda\varphi(w)}\right)}
      {(1-\lambda\varphi(u))(1-\lambda\varphi(v))}.
\]
The assigned values of $\varphi$ at $u$ and $v$ occur in the two
denominator factors, independently of the integral in the numerator.
\end{lemma}

\begin{proof}
The identity has zero integral kernel. Composition with $D$ gives,
for every integer $j\geq0$ and $u<v$,
\[
 K_{D^{j+1}}(u,v)=\varphi(u)K_{D^j}(u,v)
       +t\varphi(v)^j+t\int_u^v K_{D^j}(w,v)\,dw.
\]
The middle term integrates the evaluation term of $D^j$; the last
comes from reversing two integrations of a bounded piecewise continuous
integrand on a finite triangle. This recurrence fixes the kernel at every $u<v$,
including assigned endpoint values.

Fix a nonzero real $t$. Induction in the recurrence gives
$|K_{D^j}(u,v)|\leq j|t|(1+|t|B)^{j-1}$ for $j\geq1$, uniformly
on $0\leq u<v\leq B$. Hence, for real
$0<|\lambda|<(1+|t|B)^{-1}$,
$\sum_{j=1}^{\infty}\lambda^jK_{D^j}(u,v)$ converges uniformly.
Multiplying the recurrence by $\lambda^{j+1}$ and summing over
$j\geq0$ shows that this series, divided by $\lambda t$, satisfies
\[
 R_\lambda(u,v)=
 \frac{1}{(1-\lambda\varphi(u))(1-\lambda\varphi(v))}
 +\frac{\lambda t}{1-\lambda\varphi(u)}
                  \int_u^v R_\lambda(w,v)\,dw.
\]
For fixed $v$, the continuous function
$Y(u)=\int_u^v R_\lambda(w,v)\,dw$ satisfies, on each open partition
interval,
\[
 Y'(u)+\frac{\lambda t}{1-\lambda\varphi(u)}Y(u)
   =-\frac{1}{(1-\lambda\varphi(u))(1-\lambda\varphi(v))},
 \qquad Y(v)=0.
\]
The integrating factor
$\exp(\int_0^u\lambda t/(1-\lambda\varphi(w))\,dw)$ gives
\[
 \frac{1}{1-\lambda\varphi(v)}+\lambda tY(u)
 =\frac{\exp(\lambda t\int_u^v
                     dw/(1-\lambda\varphi(w)))}
        {1-\lambda\varphi(v)}.
\]
Continuity of $Y$ joins the solutions at partition points.
Substitution into the integral equation gives the stated $R_\lambda$,
with the assigned endpoint values. Comparing coefficients of
$\lambda^{s+1}$ proves the identity for nonzero real $t$.
Each coefficient identity is polynomial in $t$, so it holds for
formal $t$, including $t=0$.
\end{proof}

Boundary corrections require the integral and endpoint data in \cref{lem:resolvent} to vary independently. The following functional
retains them and specializes to $\Gamma_v$.

Let $w:[0,1]\longrightarrow[0,1]$ be constant on each open interval
of a finite partition, and let $\alpha,\beta\in[0,1]$. For
$G(t,X)=\sum_{s=0}^{d}h_s(t)X^s$, define
\[
 \Gamma_{w,\alpha,\beta}(G)
 =\sum_{s=0}^{d}h_s(t)[\lambda^s]
   \frac{\exp\!\left(t\lambda\int_0^1
                         \frac{w(u)}{1-\lambda u}\,du\right)}
        {(1-\lambda\alpha)(1-\lambda\beta)}.
\]
The values of $w$ at its finitely many partition points do not affect
this integral. If an interval $[a,b]\subseteq[0,B]$ with $a<b$
and a profile $\varphi$ satisfying the preceding finite partition
assumptions obey
$\int_a^b f(\varphi(u))\,du=\int_0^1 f(u)w(u)\,du$
for every continuous $f$, with $\varphi(a)=\alpha$ and
$\varphi(b)=\beta$, \cref{lem:resolvent} gives
\[
 \Gamma_{w,\alpha,\beta}(X^s)=\frac{K_{D^{s+1}}(a,b)}t
 \quad(s\geq0),\qquad
 \Gamma_{w,\alpha,\beta}(G)=\frac{K_{G(t,D)D}(a,b)}t.
\]
The extra $D$ occurs because $X^s$ represents the kernel of $D^{s+1}$.
The interval has length $\int_0^1w(u)\,du$, with no normalization.
Each coefficient depends on finitely many moments
$\int_0^1u^jw(u)\,du$ and polynomially on $\alpha,\beta$;
in particular, $\Gamma_{0,\alpha,0}(G)=G(t,\alpha)$.
The root-summation functional is the specialization
\begin{equation}\label{eq:gamma-specialization}
 \Gamma_v(G)=\Gamma_{\mathbf1_{[1-v,1]},\,1,\,1-v}(G).
\end{equation}
Indeed, coefficientwise integration of the geometric series gives
\[
 \lambda\int_{1-v}^{1}\frac{du}{1-\lambda u}
       =\log\frac{1-\lambda(1-v)}{1-\lambda}.
\]
Exponentiation and the two endpoint factors give the defining
expression for $\Gamma_v$.

If $\deg_t h_s\leq d-s$, then
$\deg_t\Gamma_{w,\alpha,\beta}(G)\leq d$, since the coefficient
$[\lambda^s]$ in its defining expression has degree at most $s$ in $t$.
For the uniform example, where $G_e(t,X)=t+2-X$, the coefficient of
$\lambda$ in the defining expression is
$t\int_0^1w(u)\,du+\alpha+\beta$. Consequently
$\Gamma_{w,\alpha,\beta}(G_e)
=(1-\int_0^1w(u)\,du)t+2-\alpha-\beta$.
Both coefficients are nonnegative because $0\leq w\leq1$ and
$\alpha,\beta\in[0,1]$, despite the negative coefficient of $X$ in $G_e$.

\subsection{The positive coordinate operator}

A unit shift represents the inclusion-exclusion index in $G_{p,k}$.
A period-uniform profile makes this shift commute with the integral
operator and fixes the resolvent contribution of each full period.

\begin{definition}
A function $m:\R\longrightarrow[0,1]$ is called
\emph{period-uniform} if it has period one, is affine on the open
intervals of a finite partition of each period, and satisfies
$\int_0^1 f(m(u))\,du=\int_0^1 f(v)\,dv$
for every continuous function $f$ on $[0,1]$. Values at the partition
points may be assigned independently. On locally bounded functions
which are piecewise continuous with
finitely many pieces on each compact interval and vanish sufficiently
far to the right, define
\[
 (\EE g)(u)=m(u)g(u)+t\int_u^\infty g(v)\,dv,
 \qquad (Sg)(u)=g(u+1).
\]
The integral is finite because $g$ vanishes to the right.
The operator for one bounded coordinate is
$\Aop=\EE+(\Id-\EE)S$.
\end{definition}

Periodicity gives $S\EE=\EE S$. The subtraction in
$\Aop=\EE-(\EE-\Id)S$ removes the shifted contribution of a
coordinate exceeding its upper bound. Cancellation of the integrals
over $[u+1,\infty)$ gives
\[
 (\Aop g)(u)=m(u)g(u)+(1-m(u))g(u+1)
                              +t\int_0^1g(u+x)\,dx.
\]
The three terms describe a zero jump with weight $m(u)$, a unit jump
with weight $1-m(u)$, and an integrated jump in $(0,1)$ with factor
$t$. Zero and unit jumps are called \emph{frozen jumps}. The open
interval specifies the kernel convention and gives the same operator
integral as $[0,1]$. The identity with the uniform rooted count is
proved in \cref{lem:uniform-circle}.

For an integer $q\geq1$, expand $\Aop^q$ by choosing one of these
three terms at each position. In terms containing an integrated jump,
replace the last such jump by the final point to define
$K_{\Aop^q}(u,v)$ at noninteger displacements. Retain all assigned
frozen weights and record integer evaluation terms separately.

\begin{lemma}\label{lem:coordinate-kernels}
For every integer $q\geq1$ and $u<v$ with $v-u\notin\Z$, we have 
\[
 \frac{K_{\Aop^q}(u,v)}t
 =\sum_{h=1}^{q}c_{q,h}(u,v)t^{h-1},
 \qquad c_{q,h}(u,v)\geq0\quad(1\leq h\leq q).
\]
\end{lemma}

\begin{proof}
Terms without integrated jumps evaluate at integer translates of $u$,
so do not contribute at a noninteger displacement. Fix a term with
$h$ integrated jumps and $j$ unit jumps, where $1\leq h\leq q$.
It is $t^h$ times an integral over $(0,1)^h$ with integrand
$\omega(x_1,\ldots,x_h)g(u+j+\sum_{i=1}^{h}x_i)$, where
$\omega\in[0,1]$ is the product of frozen weights.

For $h\geq2$, the final-point equation gives
$x_h=v-u-j-\sum_{i=1}^{h-1}x_i$. The remaining variables range over
the polytope $\Delta_{h,j}(u,v)$ defined by $0<x_i<1$ for
$1\leq i\leq h-1$ and
$0<v-u-j-\sum_{i=1}^{h-1}x_i<1$.
The change from $x_h$ to $v$ has derivative one, giving the kernel
contribution
\[
 t^h\int_{\Delta_{h,j}(u,v)}
 \omega\!\left(x_1,\ldots,x_{h-1},
                  v-u-j-\sum_{i=1}^{h-1}x_i\right)
 \,dx_1\cdots dx_{h-1}.
\]
For $h=1$, it is $t\omega(v-u-j)$ when $0<v-u-j<1$, and zero
otherwise. Assigned weights at the initial, final, and frozen
intermediate points are retained.

Each intermediate point is affine in the remaining variables.
Constant points retain their assigned $m$-values; at nonconstant
points, the affine expression for $m$ changes only on finitely
many proper affine hyperplanes. The integrands are bounded and
piecewise polynomial on polytopes, hence Riemann integrable, also
as successive integrals.

Each contribution is $t^h$ times a nonnegative quantity. Summing
over terms with $h$ integrated jumps defines $c_{q,h}(u,v)$;
division by $t$ proves the formula.
\end{proof}

Restrictions on sums of consecutive jumps add linear inequalities
to these integration domains. The restricted kernels still have
nonnegative coefficients; the allowed evaluation terms are retained
separately.

\begin{example}
For a noninteger displacement with $0<v-u<2$,
\[
 \frac{K_{\Aop^2}(u,v)}t
 =\begin{cases}
 m(u)+m(v)+t(v-u),&0<v-u<1,\\
 2-m(u)-m(v)+t(2-v+u),&1<v-u<2.
 \end{cases}
\]
One integrated jump is accompanied by a zero jump in the first range,
giving $m(u)+m(v)$, or a unit jump in the second, giving
$(1-m(u))+(1-m(v-1))=2-m(u)-m(v)$ by periodicity.
Two integrations contribute $t^2$ times the length of
$(0,1)\cap(v-u-1,v-u)$, namely $v-u$ or $2-v+u$.
\end{example}

For any finite product below, record terms without integration as
evaluations at nonnegative integer translates, and define kernels by
eliminating the last integrated jump. For polynomials in $\EE,S$, \cref{lem:resolvent} supplies the kernels of powers of $\EE$,
while $S^j$ shifts the starting point by $j$. The kernel recurrence involves
only the interval between the initial and final points, so truncating
the upper integration limit at any $B>v$ leaves the kernel at $(u,v)$ unchanged.
Evaluation coefficients retain all assigned values of $m$.

\begin{lemma}
\label{lem:kernel-compatibility}
Consider finite evaluation-and-integral expansions formed from
$\EE,S,\Aop$, including restrictions of coordinate blocks by
linear inequalities on their jumps. Allow also evaluation terms
$a(t,u)g(u+j)$, where $j$ is a nonnegative integer and $a$ is a
polynomial in $t$ whose coefficients are piecewise polynomial in
$u$, with finitely many pieces on each compact interval and
specified values at partition points.

Suppose that two representations of each factor have the same
coefficient of every integer-translate evaluation at every
starting point, and have the same integral kernel at every
noninteger displacement. Their finite compositions then have the
same evaluation coefficients and the same integral kernels at
every noninteger total displacement.
\end{lemma}

\begin{proof}
All calculations are coefficientwise in $t$ and involve finitely many
powers. Fix $u$ and $g$ in the operator domain, and choose $B>u$ with $g=0$ on
$[B,\infty)$. All jumps are nonnegative, so a contributing term
has total jump at most $B-u$. This bounds every integrated jump
and truncates all integrals in a finite composition, including
those defining $\EE$. Intermediate outputs also vanish sufficiently
far to the right.

The jump variables now range over finite polytopes. Arguments of
profiles and evaluation coefficients are affine: nonconstant arguments
meet partition points on finitely many proper affine hyperplanes,
and constant arguments retain their specified values. Coefficient
integrands are therefore piecewise polynomial on finite polyhedral
partitions. Multiplying by $g$ at the final point gives bounded
piecewise continuous integrands. This remains true after fixing any
subset of variables, with assigned values retained at constant arguments.
Thus the multiple integrals and all their sections are Riemann
integrable, allowing successive integrations to be reversed. At fixed
initial and final points, the same reasoning applies after eliminating
the last integrated jump.

It suffices to consider two factors, written as
\[
 (Tg)(u)=\sum_{i=0}^{I}a_i(t,u)g(u+i)
                  +\int_u^\infty K_T(u,z)g(z)\,dz
\]
and
\[
 (Ug)(u)=\sum_{j=0}^{J}b_j(t,u)g(u+j)
                  +\int_u^\infty K_U(u,z)g(z)\,dz.
\]
Kernels are set to zero outside their forward support.
The evaluation part of $TU$ is
\[
 \sum_{i=0}^{I}\sum_{j=0}^{J}
       a_i(t,u)b_j(t,u+i)g(u+i+j).
\]
An evaluation before an integral shifts its starting point to $u+i$;
an evaluation after an integral uses $v=z+j$, with derivative one.
Two integrals range over $u\leq z\leq v$. Thus a representative of
the composed kernel is
\begin{align*}
K_{TU}(u,v)=\sum_{i=0}^{I}a_i(t,u)K_U(u+i,v)+\sum_{j=0}^{J}K_T(u,v-j)b_j(t,v-j)+\int_u^vK_T(u,z)K_U(z,v)\,dz.
\end{align*}
This formula also represents the grouped jump expansion at noninteger
total displacement. If both groups contain integrated jumps, replace
their last integrated jumps by $z$ and $v$, respectively. Both changes
have derivative one; reversing the remaining finite Riemann integrals
gives the last term. Changes of kernel convention at integer block
displacements affect only finitely many intermediate $z$ and leave
this integral unchanged. If just one group contains an integrated jump,
the first or second sum results. Its displacement is noninteger because
the other group's displacement is integral. Terms without integration
give the displayed evaluation part.

Fix $u<v$ with $v-u\notin\Z$. Each supported kernel in the two
sums has noninteger displacement, since $i,j$ are integers. Thus
the sums agree for the two representations, with evaluation
coefficients taken at their specified points. In the integral, the
integrands agree except possibly where $z-u$ or $v-z$ is an integer.
There are finitely many such points in $[u,v]$, so the integrals agree.
Evaluation coefficients agree pointwise as well. Induction proves the
claim for finite compositions.
\end{proof}

For $\Aop=\EE+(\Id-\EE)S$, tail cancellation gives evaluation
weights $m(u),1-m(u)$ at jumps $0,1$, and kernel $t$ on
$0<v-u<1$, zero at other noninteger displacements. Hence its polynomial
and jump representations satisfy \cref{lem:kernel-compatibility}.
Periodicity, including assigned values at integer translates, gives
$S\EE=\EE S$ in both. The lemma therefore identifies their finite
compositions at noninteger total displacement and permits insertion of
a restricted block with its explicitly computed evaluation correction
between external coordinate operators. Integer-displacement kernel
values remain distinct from evaluation coefficients.

\subsection{The complementary profile}

The profile used in the transfer identity is chosen to match the
coefficient formula for $\Gamma_{w,\alpha,\beta}(f_b)$.
For a general $T$, linearity and coefficient convolution give
\[
 \begin{aligned}
 \Gamma_{w,\alpha,\beta}(f_b(T,X))
 &=\sum_{s=0}^{b-1}(-1)^s\binom{T+b}{b-1-s}
                         \Gamma_{w,\alpha,\beta}(X^s)\\
 &=[v^{b-1}](1+v)^{T+b}
       \sum_{s=0}^{\infty}(-v)^s\Gamma_{w,\alpha,\beta}(X^s)\\
 &=[v^{b-1}]
   \frac{(1+v)^{T+b}}{(1+v\alpha)(1+v\beta)}
   \exp\!\left(-tv\int_0^1\frac{w(u)}{1+vu}\,du\right).
 \end{aligned}
\]
Only the terms with $0\leq s\leq b-1$ contribute to the extracted
coefficient. The last equality evaluates the formal series defining
$\Gamma$ at $\lambda=-v$.
The substitution $v=\lambda/(1-\lambda)$ changes this to
\[
 [\lambda^{b-1}]
 \frac{(1-\lambda)^{-T}
       \exp\!\left(-t\lambda\int_0^1
                 \frac{w(u)}{1-\lambda(1-u)}\,du\right)}
      {(1-\lambda(1-\alpha))(1-\lambda(1-\beta))}.
\]

Set $L=\int_0^1w(u)\,du\in(0,1)$ and
$T=(r-j)t-\delta$, where $r\geq1$, $0\leq j\leq r-1$, and
$0\leq\delta\leq b-1$ are integers. The factor
$(1-\lambda)^{-(r-j)t}$ represents $r-j$ uniform periods, and the
negative exponential removes the weight $w(1-u)$. We therefore seek
a period-uniform profile with $m(j)=1-\alpha$, $m(r-L)=1-\beta$, and
\[
 \int_j^{r-L}f(m(u))\,du
  =(r-j)\int_0^1f(v)\,dv-\int_0^1f(1-u)w(u)\,du
\]
for every continuous $f$ on $[0,1]$. Writing $w^\vee(u)=w(1-u)$,
construct the initial arc of length $1-L$ with weight $1-w^\vee$
and the remaining arc with weight $w^\vee$:
\[
 \int_0^{1-L}f(m(u))\,du
     =\int_0^1 f(v)(1-w^\vee(v))\,dv,
 \qquad
 \int_{1-L}^{1}f(m(u))\,du
     =\int_0^1 f(v)w^\vee(v)\,dv.
\]
On each partition interval $[a,b]$ where $w^\vee=d$, allot
subintervals of lengths $(1-d)(b-a)$ and $d(b-a)$ in the initial
and remaining arcs, respectively, and let the profile traverse $[a,b]$
linearly on each. Omit zero-length subintervals and concatenate in partition
order. Change of variables gives $(1-d)\int_a^b f(v)\,dv$ and
$d\int_a^b f(v)\,dv$. Summing proves the two identities and
period-uniformity; extend with period one.
Assign $m(0)=1-\alpha$ and $m(1-L)=1-\beta$ at all integer
translates. These finite changes per period leave the integrals unchanged.

For example, if $w(u)=c$ with $0<c<1$, then $L=c$ and the
profile on the open parts of one period is
\[
 m(u)=
 \begin{cases}
 u/(1-c),&0<u<1-c,\\
 (u-1+c)/c,&1-c<u<1.
 \end{cases}
\]
The assignments $m(0)=1-\alpha$ and $m(1-c)=1-\beta$ are
retained, and the profile is extended with period one.
Ordinary change of variables gives
$\int_0^{1-c}f(m(u))\,du=(1-c)\int_0^1f(v)\,dv$
and $\int_{1-c}^{1}f(m(u))\,du=c\int_0^1f(v)\,dv$.

For every integer
$0\leq j\leq r-1$, the interval $[j,r-L]$ consists of $r-j-1$
full periods and the initial arc. Its integral therefore gives
\[
 R_\lambda(j,r-L)=
 \frac{(1-\lambda)^{-(r-j)t}
       \exp\!\left(-t\lambda\int_0^1
                  \frac{w(u)}{1-\lambda(1-u)}\,du\right)}
      {(1-\lambda(1-\alpha))(1-\lambda(1-\beta))}.
\]
This follows from
$\lambda\int_0^1du/(1-\lambda u)=-\log(1-\lambda)$ and the
substitution $v=1-u$ in the reflected term.

\begin{lemma}\label{lem:transfer}
Let $r\geq1$, $b\geq1$, $0\leq j\leq r-1$, and
$0\leq\delta\leq b-1$ be integers. For the complementary profile just constructed,
\[
 \Gamma_{w,\alpha,\beta}
       \bigl(f_b((r-j)t-\delta,X)\bigr)
 =\frac{K_{S^j\EE^{b-\delta}(\EE-\Id)^\delta}(0,r-L)}t.
\]
Moreover, for $t\in\Z_{\geq1}$ and $a\in\{0,\ldots,t\}$, we have 
\[
 [z^{rt-a}]\frac{z^{jt+\delta}}{(1-z)^b}
       =\binom{(r-j)t-a-\delta+b-1}{b-1}.
\]
The right side is interpreted as a binomial polynomial, including
when the remaining exponent is negative.
\end{lemma}

\begin{proof}
For $T=(r-j)t-\delta$, the coefficient formula above is
$[\lambda^{b-1}](1-\lambda)^\delta R_\lambda(j,r-L)$.
Since $\delta\leq b-1$, every exponent $b-\delta+\ell$ below
is at least one. The binomial expansion and \cref{lem:resolvent} therefore give
\[
 \begin{aligned}
 \frac{K_{\EE^{b-\delta}(\EE-\Id)^\delta}(j,r-L)}t
 &=\sum_{\ell=0}^{\delta}(-1)^{\delta-\ell}\binom{\delta}{\ell}
       \frac{K_{\EE^{b-\delta+\ell}}(j,r-L)}t\\
 &=\sum_{\ell=0}^{\delta}(-1)^{\delta-\ell}\binom{\delta}{\ell}
       [\lambda^{b-\delta+\ell-1}]R_\lambda(j,r-L)\\
 &=[\lambda^{b-1}](1-\lambda)^\delta R_\lambda(j,r-L).
 \end{aligned}
\]
The shift $S^j$ changes the starting point from $0$ to $j$, which
proves the kernel identity.

The remaining counting exponent is $N=(r-j)t-a-\delta$.
For $N\geq0$, extracting the coefficient of $(1-z)^{-b}$ gives
the stated binomial coefficient. For $N<0$,
$N\geq(r-j-1)t-\delta\geq-\delta\geq-(b-1)$, so
$0\leq N+b-1<b-1$. Both the binomial polynomial and the
extracted coefficient are then zero, including at small positive $t$.
\end{proof}

The transfer lemma applies to the specified coefficient and its
corresponding kernel. It does not define a substitution homomorphism
on arbitrary rational generating functions.

\begin{lemma}\label{lem:uniform-circle}
Let $p\geq1$ and $1\leq k\leq p$ be integers. For the complementary profile associated with $w,\alpha,\beta$ and $0<L<1$, we have 
\[
 \Gamma_{w,\alpha,\beta}(G_{p,k})
                       =\frac{K_{\Aop^p}(0,k-L)}t.
\]
In particular, this polynomial has nonnegative coefficients in $t$.
\end{lemma}

\begin{proof}
Commutativity of $S$ and $\EE$ gives
\[
 \Aop^p=\sum_{j=0}^{p}\binom pj
                        S^j\EE^{p-j}(\Id-\EE)^j.
\]
At total jump $k-L$, the terms with $j\geq k$ cannot reach the
endpoint and contribute zero. For $0\leq j\leq k-1$, the transfer
lemma, with $r=k$, $b=p$, and $\delta=j$, gives the $j$-th term in the
transform of the uniform rooted count. The factor
$(\Id-\EE)^j=(-1)^j(\EE-\Id)^j$ gives the inclusion-exclusion sign. \cref{lem:kernel-compatibility} identifies this polynomial
kernel with the coordinate jump kernel at the noninteger displacement
$k-L$. Coefficient nonnegativity follows from \cref{lem:coordinate-kernels}.
\end{proof}

\section{Hyperplane constraints and coefficient nonnegativity}\label{sec:constraints}

\subsection{An upper interval sum}

For integers $p\geq1$, $1\leq k\leq p$, and $t\geq1$, the generating
polynomial for a block sum at most $kt$ is
\[
 B_{p,k}(t,z)=\sum_{x_1=0}^{t}\cdots\sum_{x_p=0}^{t}
       \mathbf1_{\{x_1+\cdots+x_p\leq kt\}}z^{x_1+\cdots+x_p}.
\]

\begin{lemma}\label{lem:upper-discrete}
The block generating polynomial satisfies
\[
 B_{p,k}(t,z)=\sum_{j=0}^{k-1}(-1)^j\binom pj
                   \frac{z^{j(t+1)}}{(1-z)^p}
       -\frac{z^{kt+1}}{1-z}
                   G_{p,k}\!\left(t,-\frac z{1-z}\right).
\]
\end{lemma}

\begin{proof}
Through degree $kt$, the first sum counts bounded compositions:
terms with at least $k$ coordinates exceeding $t$ begin in degree
$k(t+1)>kt$. For $v\geq0$, its coefficient in degree $kt+1+v$ is
$F_{p,k}(-v-1,t)$. The binomial expansion of $F_{p,k}$ gives
\[
 F_{p,k}(-v-1,t)=\sum_{s=0}^{p-1}h_s(t)(-1)^s\binom vs,
 \qquad G_{p,k}(t,X)=\sum_{s=0}^{p-1}h_s(t)X^s.
\]
Here $\binom{-v-1+s}{s}=(-1)^s\binom vs$, so the tail is
\[
 \sum_{v=0}^{\infty}F_{p,k}(-v-1,t)z^v
 =\sum_{s=0}^{p-1}h_s(t)(-1)^s\frac{z^s}{(1-z)^{s+1}}
 =\frac1{1-z}G_{p,k}\!\left(t,-\frac z{1-z}\right).
\]
The upper binomial arguments in the tail are at least $p-j\geq1$
(the minimum occurs at $v=0$, $a=-1$), including at small dilations.
Subtracting the tail proves the identity.
\end{proof}

The block operator corresponding to \cref{lem:upper-discrete} is
\[
 \BB_{p,k}=
 \sum_{j=0}^{k-1}\binom pj S^j\EE^{p-j}(\Id-\EE)^j
       +S^k(\Id-\EE)G_{p,k}(t,\Id-\EE).
\]
\cref{lem:rooted} verifies the transfer term by term after including
all external coordinate factors. The following correction measures
the change in the evaluation weight at jump $k$.

\begin{lemma}\label{lem:upper-operator}
The integral kernel of $\BB_{p,k}$ at noninteger block jumps and
its integer evaluation terms agree with the expansion of $\Aop^p$
restricted to block jumps at most $k$, with the additional evaluation term
$Q_{p,k}(t,m(u))g(u+k)$, where
\[
 Q_{p,k}(t,x)=(1-x)G_{p,k}(t,1-x)-\binom pk(1-x)^kx^{p-k}.
\]
Every coefficient in $t$ of $Q_{p,k}(t,x)$ is nonnegative for
$0\leq x\leq1$. A restriction retains eligible zero and unit steps
and imposes the block inequality on each integration domain.
\end{lemma}

\begin{proof}
\emph{Kernel cancellation and evaluations below $k$.}
Below jump $k$, the last term of $\BB_{p,k}$ vanishes, and its first
sum contains exactly the terms of $\Aop^p$ that can reach the jump.
This includes the evaluation weights at $0,\ldots,k-1$.

For cancellation above $k$, the coefficient formula for $f_p$ gives,
for any polynomial parameter $T$,
\[
 (\Id-\EE)f_p(T,\Id-\EE)
 =[\lambda^{p-1}](1-\lambda)^{-T-1}
                  (\Id-\EE)(\Id-\lambda\EE)^{-1}.
\]
The integral kernel of the expression on the right, divided by
$t$, is $-[\lambda^{p-1}](1-\lambda)^{-T}R_\lambda(u,v)$.
Indeed, the kernel of $(\Id-\lambda\EE)^{-1}$ is
$\lambda tR_\lambda$, whereas that of
$\EE(\Id-\lambda\EE)^{-1}$ is $tR_\lambda$.

Suppose $v-u>k$. For $0\leq j\leq k-1$, the kernel of
$S^j\EE^{p-j}(\Id-\EE)^j$, divided by $t$, is
\[
 [\lambda^{p-1}](-1)^j(1-\lambda)^jR_\lambda(u+j,v).
\]
The arcs $[u+j,v]$ and $[u+k,v]$ differ by $k-j$ full periods, so
$R_\lambda(u+j,v)=(1-\lambda)^{-(k-j)t}R_\lambda(u+k,v)$.
The $j$-th term of the first sum therefore has kernel divided by $t$ equal to
$\binom pj[\lambda^{p-1}](-1)^j
(1-\lambda)^{j-(k-j)t}R_\lambda(u+k,v)$.
The $j$-th term of $S^k(\Id-\EE)G_{p,k}(t,\Id-\EE)$
has the negative of this expression, by the preceding kernel identity
with $T=(k-j)t-j$. Thus the kernel above $k$ is zero.

\smallskip\noindent
\emph{The evaluation at jump $k$.}
The first sum has shifts smaller than $k$, so the evaluation weight
at $k$ is the multiplication part of the last term,
$(1-m(u))G_{p,k}(t,1-m(u))$. In $\Aop^p$, exactly $k$ unit steps and
$p-k$ zero steps give the old weight
$\binom pk(1-m(u))^km(u)^{p-k}$, since periodicity makes all endpoint
values equal to $m(u)$. Their difference is $Q_{p,k}(t,m(u))$.
Neither $\BB_{p,k}$ nor the restricted expansion of $\Aop^p$
has evaluation terms above $k$.
\cref{lem:kernel-compatibility} permits composition with external
coordinate operators at noninteger total displacement.

\smallskip\noindent
\emph{Coefficients of the correction.}
The subtracted evaluation weight has degree zero in $t$, so only
the constant coefficient requires a comparison. Set $s=1-x$ with
$0<s<1$. In \cref{lem:uniform-circle}, choose $w(u)=\varepsilon$,
$\alpha=s$, and $\beta=0$, where $0<\varepsilon<1$.
The total jump is $k-\varepsilon$; the complementary profile has
endpoints $x,1$ and
$\int_0^{1-\varepsilon}g(m(u))\,du
=(1-\varepsilon)\int_0^1g(y)\,dy$.
The operator $\Aop$ below depends on this profile.
As $\varepsilon\downarrow0$, the finite moments defining $\Gamma$
tend to zero, while the endpoints remain fixed. Consequently,
\[
 sG_{p,k}(t,s)
   =s\lim_{\varepsilon\downarrow0}
                        \frac{K_{\Aop^p}(0,k-\varepsilon)}t.
\]
The coefficients on the right are nonnegative, so all positive-degree
coefficients of $sG_{p,k}(t,s)$ are nonnegative.

For the constant coefficient, the contributing paths before division
by $t$ have one continuous step and $k-1$ unit steps.
The continuous jump must be $1-\varepsilon$. If this step has
index $i$, all unit steps must occur before it, since a later unit
step has weight $1-1=0$. There are $\binom{i-1}{k-1}$ choices for
these earlier unit steps. Every zero step before the continuous step has weight $x=1-s$,
every earlier unit step has weight $s$, and every zero step after
it has weight $1$. Including the factor $s$ therefore gives
\[
 [t^0]\bigl(sG_{p,k}(t,s)\bigr)
   =s^k\sum_{i=k}^{p}\binom{i-1}{k-1}(1-s)^{i-k}.
\]
This counts length-$p$ binary strings with at least $k$ ones,
weighted by $s$ for a one and $1-s$ for a zero: $i$ is the position
of the $k$-th one, and the suffix has total weight
$(s+(1-s))^{p-i}=1$. Grouping by the number of ones gives
\[
 [t^0]\bigl(sG_{p,k}(t,s)\bigr)
       =\sum_{j=k}^{p}\binom pj s^j(1-s)^{p-j}.
\]
Subtracting the term with exactly $k$ ones yields
\[
 [t^0]Q_{p,k}(t,x)
       =\sum_{j=k+1}^{p}\binom pj(1-x)^jx^{p-j}\geq0,
\]
with an empty sum when $k=p$. This proves all coefficients are
nonnegative for $0<x<1$; polynomial continuity gives $x=0,1$.
\end{proof}

\begin{example}[\cref{ex:four-element}, continued]
The inequality $x_1+x_2\leq t$ has block parameters $p=2$ and $k=1$.
The corresponding uniform transform and block operator are
\[
 G_{2,1}(t,X)=t+2-X,
 \qquad
 \BB_{2,1}=\EE^2+S(\Id-\EE)((t+1)\Id+\EE).
\]
At block jump one, the two frozen paths have total weight
$2m(u)(1-m(u))$, whereas $\BB_{2,1}$ has weight
$(1-m(u))(t+1+m(u))$. Their difference is
\[
 Q_{2,1}(t,m(u))=(1-m(u))(t+1-m(u)).
\]
This weight at $u+1$ supplements the original paths with block
jumps at most one.
\end{example}

\subsection{A lower interval sum}

For a hyperplane containing the root, the complementary sum is
at least $t$. The forbidden tuples on the complementary block have
sum at most $t-1$, so their coordinate upper bounds are inactive.
These tuples have generating polynomial
\[
 D_p(t,z)=\sum_{v=0}^{t-1}\binom{v+p-1}{p-1}z^v
 =\frac1{(1-z)^p}
       -\frac{z^t}{1-z}f_p\!\left(t-1,-\frac z{1-z}\right).
\]
The second equality follows by summing the tail starting in degree
$t$, as in \cref{lem:upper-discrete}. The associated polynomial
operator is $\DD_p=\EE^p-S\EE f_p(t-1,\Id-\EE)$.

\begin{lemma}\label{lem:lower-operator}
As an operator, $\DD_p$ is represented by the expansion of $\Aop^p$ restricted to block
jumps less than one, with the additional evaluation term
$-Q_p^{\mathrm{low}}(t,m(u))g(u+1)$, where
$Q_p^{\mathrm{low}}(t,x)=xf_p(t-1,1-x)$.
The integral kernels are compared at noninteger block jumps;
integer evaluation terms are recorded separately.
The polynomial $Q_p^{\mathrm{low}}(t,x)$ has nonnegative coefficients
in $t$ for $0\leq x\leq1$.
\end{lemma}

\begin{proof}
Below jump one, all shifted terms vanish, so $\DD_p$, $\EE^p$, and
$\Aop^p$ have the same kernel and evaluation terms. For $v-u>1$,
the coefficient formula for $f_p$ gives
\[
 \frac{K_{\EE f_p(t-1,\Id-\EE)}(u+1,v)}t
       =[\lambda^{p-1}](1-\lambda)^{-t}R_\lambda(u+1,v).
\]
A full uniform period gives
$(1-\lambda)^{-t}R_\lambda(u+1,v)=R_\lambda(u,v)$,
so this kernel cancels that of $\EE^p$ above jump one.
At jump one, the shifted term has evaluation weight
$-m(u)f_p(t-1,1-m(u))$; at zero, the weight agrees with $\EE^p$.
There are no other evaluation terms.
\cref{lem:kernel-compatibility} permits insertion between external
coordinate operators at noninteger total displacement.

The coefficient representation for $f_p$ yields
\[
 f_p(t-1,1-x)
 =[\lambda^{p-1}]\frac{(1-\lambda)^{-t}}{1-\lambda x}
 =\sum_{j=0}^{p-1}x^{p-1-j}\binom{t+j-1}{j}.
\]
For $j=0$ the binomial factor is $1$. For $j\geq1$ it is
$t(t+1)\cdots(t+j-1)/j!$, which has nonnegative coefficients.
Multiplication by $x\geq0$ proves the assertion.
\end{proof}

\subsection{Polynomial rooted counts and exact transfer}

Let $M$ be the connected paving positroid of \cref{lem:geometry},
on $\{e,1,\ldots,q\}$ with rank $r$ and dependent hyperplanes
$H_1,\ldots,H_\ell$. Let $F_i^{\mathrm{forb}}(a,t)$ count the points
in the uniform rooted slice violating the inequality for $H_i$.
Their disjointness gives, for integers $t\geq1$ and $0\leq a\leq t$,
\[
 F_e(a,t)=F_{q,r}(a,t)-\sum_{i=1}^{\ell}F_i^{\mathrm{forb}}(a,t).
\]
This is the rooted version of the established paving cap subtraction
\cite[Proposition~5.10 and Theorem~5.11]{HanelyEtAl2023}.

The block associated with $H_i$ is $H_i$ when $e\notin H_i$,
and $E\setminus H_i$ when $e\in H_i$. If it occupies coordinates
$c,\ldots,c+p-1$, where $1\leq c\leq q-p+1$, its full inserted
forbidden operator is
\[
 \mathcal T_i^{\mathrm{forb}}=
 \begin{cases}
 \Aop^{c-1}(\Aop^p-\BB_{p,r-1})\Aop^{q-c-p+1},
       &e\notin H_i,\quad p=|H_i|,\\
 \Aop^{c-1}\DD_p\Aop^{q-c-p+1},
       &e\in H_i,\quad p=|E\setminus H_i|.
 \end{cases}
\]
The block position is determined by the fixed cyclic order.

\begin{lemma}\label{lem:rooted}
The rooted count $F_e(a,t)$ is represented on its integer triangle by
a polynomial of total degree at most $q-1$. Its binomial transform
$G_e(t,X)=(\NN F_e)(t,X)$ satisfies, for every finite piecewise
constant $w:[0,1]\to[0,1]$ with $L=\int_0^1w(u)\,du\in(0,1)$
and every $\alpha,\beta\in[0,1]$,
\[
 \Gamma_{w,\alpha,\beta}(G_e)
   =\frac1tK_{\Aop^q-\sum_{i=1}^{\ell}
                         \mathcal T_i^{\mathrm{forb}}}(0,r-L).
\]
\end{lemma}

\begin{proof}
\cref{lem:uniform-circle} gives the uniform term. For a block
of size $p$, the $q-p$ external coordinates contribute
\[
 \left(\frac{1-z^{t+1}}{1-z}\right)^{q-p}
       =\sum_{j=0}^{q-p}(-1)^j\binom{q-p}{j}
                         \frac{z^{j(t+1)}}{(1-z)^{q-p}}.
\]
The parameters in \cref{lem:transfer} are recorded below.
In the upper rows, $k=r-1$ and $0\leq v\leq k-1$.
In both boundary rows, $0\leq s\leq p-1$, and every row has
$0\leq j\leq q-p$.
\[
 \begin{array}{c|c|c|c|c}
  \text{Term}&b&\text{Period shift}&\delta&\text{Retained shift}\\ \hline
  \text{Upper main}&q&v+j&v+j&v+j<r\\
  \text{Upper boundary}&1+s+q-p&k+j&1+s+j&k+j<r\\
  \text{Lower main}&q&j&j&j<r\\
  \text{Lower boundary}&1+s+q-p&1+j&s+j&1+j<r
 \end{array}
\]
For an upper block, $p=|H_i|$, $k=r-1$, and $q-p\geq1$.
The retained count is the coefficient of $z^{rt-a}$ in the product
of $B_{p,k}(t,z)$ and the external factor. In the first part of
\cref{lem:upper-discrete}, an internal index $v$ and an external
index $j$ give the factor
\[
 \frac{z^{(v+j)t+v+j}}{(1-z)^q},
 \qquad 0\leq v\leq k-1,\quad 0\leq j\leq q-p.
\]
If $v+j<r$, \cref{lem:transfer} applies with
$b=q$ and constant shift $\delta=v+j\leq q-1$.
If $v+j\geq r$, the remaining exponent is negative, since
$v+j\geq1$, and the corresponding periodic shift cannot reach $r-L$.

The boundary polynomial is
$G_{p,k}(t,X)=\sum_{s=0}^{p-1}h_s(t)X^s$.
Multiplying the tail in \cref{lem:upper-discrete} by the
external factor gives the $(s,j)$-term
\[
 (-1)^{s+j+1}h_s(t)\binom{q-p}{j}
       \frac{z^{(k+j)t+1+s+j}}{(1-z)^{1+s+q-p}}.
\]
The signs come from subtracting the tail, substituting
$-z/(1-z)$ in $G_{p,k}$, and expanding the external factor.
A retained shift satisfies $k+j<r$, hence $j=0$ because $k=r-1$.
Consequently $\delta=1+s\leq s+q-p=b-1$, as $q-p\geq1$.
\cref{lem:transfer} identifies the transformed retained term
with $t^{-1}$ times the kernel at $(0,r-L)$ of
\[
 (-1)^{s+1}h_s(t)S^k\EE^{q-p}(\EE-\Id)^{s+1}
       =h_s(t)S^k\EE^{q-p}(\Id-\EE)^{s+1}.
\]
Terms with $j\geq1$ have negative remaining exponents because
$k+j\geq r$ and their constant shifts are at least one.
The corresponding operator summands below also have zero kernel
at $(0,r-L)$: each contains $S^{k+j}$, with $k+j\geq r>r-L$,
and all remaining jumps are nonnegative. Adding these zero
contributions and using the binomial expansion of $\Aop^{q-p}$ gives
\begin{align*}
\sum_{s=0}^{p-1}\sum_{j=0}^{q-p}
 h_s(t)\binom{q-p}{j}S^{k+j}\EE^{q-p-j}(\Id-\EE)^{s+j+1}=-S^k(\EE-\Id)G_{p,k}(t,\Id-\EE)\Aop^{q-p}.
\end{align*}
This is the boundary part of $\BB_{p,k}$ times the external operator.
Its factors can be placed at the original block, since the complete
polynomials in $\EE,S$ commute. This commutation concerns the
complete operator identity; it does not rearrange the individual
path integration domains. \cref{lem:kernel-compatibility}
identifies the resulting polynomial and block kernels at the
noninteger total displacement $r-L$.

Each binomial polynomial from the first part has total degree $q-1$.
In a boundary term, $\deg_t h_s\leq p-1-s$, and the extracted
binomial polynomial has degree $s+q-p$. Their product therefore has
total degree at most $q-1$. The upper retained count and its forbidden
count consequently have polynomial representations with the required
transferred operators.

For a hyperplane containing $e$, set $p=|E\setminus H_i|$.
Here $q-p=|H_i|-1\geq r-1$. Its forbidden count is the coefficient
of $z^{rt-a}$ in the product of $D_p(t,z)$ and the external factor. The first
term is covered by the preceding verification with $v=0$.
For a boundary coefficient of $f_p(t-1,X)$ with index $s$,
the last row gives $\delta=s+j\leq s+q-p=b-1$, since $j\leq q-p$.
\cref{lem:transfer} therefore applies whenever $1+j<r$.
Whenever $1+j\geq r$, we have
$j\geq r-1\geq1$, so the remaining exponent is negative and
the periodic shift cannot reach $r-L$.
The transferred boundary part is
$-S\EE f_p(t-1,\Id-\EE)\Aop^{q-p}$, as in $\DD_p$.
Its coefficient of index $s$ has degree at most $p-1-s$, and its
binomial polynomial has degree $s+q-p$, so their product has total
degree at most $q-1$.

Subtracting the forbidden counts constructs $F_e$ and its transferred
image. Uniqueness follows because a polynomial vanishing on the integer
triangle is zero: for each fixed integer $a\geq0$ it vanishes at
infinitely many $t$, so its coefficients in $t$ vanish at every such $a$.
\end{proof}

\begin{example}[\cref{ex:four-element}, continued]
In the uniform rooted slice, the points violating $x_1+x_2\leq t$
have $t+1\leq s=x_1+x_2\leq2t-a$. There are $2t-s+1$ bounded
pairs with sum $s$, so their count is
\[
 F_H^{\mathrm{forb}}(a,t)
 =\sum_{s=t+1}^{2t-a}(2t-s+1)
 =\binom{t+1}{2}-\binom{a+1}{2},
\]
where the sum is empty when $a=t$. Its transform is
\[
 (\NN F_H^{\mathrm{forb}})(t,X)
 =\frac{t(t+1)}2+X-X^2.
\]
The constrained block is $(1,2)$, and coordinate $3$ contributes
$\Aop$. Thus
\[
 \mathcal T_H^{\mathrm{forb}}
       =(\Aop^2-\BB_{2,1})\Aop,
 \qquad
 \Aop^3-\mathcal T_H^{\mathrm{forb}}=\BB_{2,1}\Aop.
\]
For $0<L<1$, \cref{lem:rooted} gives
\[
 \Gamma_{w,\alpha,\beta}(G_e)
       =\frac{K_{\BB_{2,1}\Aop}(0,2-L)}t.
\]
For coordinate jumps $d_1,d_2,d_3$, the allowed paths have
$d_1+d_2\leq1$, and the inserted correction is
\[
 (\mathcal C_Hg)(u)
   =Q_{2,1}(t,m(u))(\Aop g)(u+1).
\]
At total jump $2-L$, its remaining coordinate must make the
continuous jump $1-L$. Since $m(0)=1-\alpha$, its normalized
kernel contribution is
\[
 \frac{K_{\mathcal C_H}(0,2-L)}t
       =Q_{2,1}(t,1-\alpha)=\alpha(t+\alpha).
\]
The external factor is therefore essential: it turns the block
evaluation into a contribution to the full integral kernel.
\end{example}

\subsection{Disjoint forbidden events}

At total jump $r-L$, the expansion of $\Aop^q$ has coordinate jumps
$d_1,\ldots,d_q\in[0,1]$ with sum $r-L$. Write
$d(A)=\sum_{j\in A}d_j$ for $A\subseteq\{1,\ldots,q\}$.
The forbidden conditions are $d(H_i)>r-1$ if $e\notin H_i$ and
$d(E\setminus H_i)<1$ if $e\in H_i$.

\begin{lemma}\label{lem:disjoint-paths}
For $0<L<1$, these forbidden conditions in the path expansion of
$\Aop^q$ are pairwise disjoint.
\end{lemma}

\begin{proof}
Set $y_e=L$ and $y_j=d_j$ for $1\leq j\leq q$.
Then $y\in[0,1]^E$ and $y(E)=L+\sum_{j=1}^{q}d_j=r$, so
$y\in\Delta(r,n)$. If $e\notin H_i$, the forbidden condition is
$y(H_i)>r-1$. If $e\in H_i$, the identity
$y(H_i)=r-d(E\setminus H_i)$ gives the same condition.
These regions are disjoint by \cref{lem:geometry}.
\end{proof}

\begin{proposition}\label{prop:kernel-positive}
For every finite piecewise constant function $w:[0,1]\to[0,1]$ and
all $\alpha,\beta\in[0,1]$, the polynomial
$\Gamma_{w,\alpha,\beta}(G_e)$ has nonnegative coefficients in $t$.
\end{proposition}

\begin{proof}
Assume first that $0<L<1$. For $1\leq i\leq\ell$, let
$\mathcal C_i$ denote the full inserted boundary correction.
Its block factor sends $g$ to the function
\[
 u\longmapsto
 \begin{cases}
 Q_{p,r-1}(t,m(u))g(u+r-1),&e\notin H_i,\quad p=|H_i|,\\
 Q_p^{\mathrm{low}}(t,m(u))g(u+1),
                         &e\in H_i,\quad p=|E\setminus H_i|.
 \end{cases}
\]
The block position and external factors are those of
$\mathcal T_i^{\mathrm{forb}}$. By \cref{lem:upper-operator,lem:lower-operator},
$\mathcal T_i^{\mathrm{forb}}$ is the original expansion restricted
to violation of constraint $i$, minus $\mathcal C_i$.

On each domain of the common expansion of $\Aop^q$, subtracting
the original forbidden terms multiplies the integrand by
\[
 1-\sum_{i=1}^{\ell}\mathbf1_{\{\text{the path violates constraint }i\}}.
\]
By \cref{lem:disjoint-paths}, this is $1$ for allowed paths
and $0$ otherwise; the same holds for zero and unit evaluation terms.
\cref{lem:kernel-compatibility} therefore gives, at $r-L$,
\[
 K_{\Aop^q-\sum_{i=1}^{\ell}\mathcal T_i^{\mathrm{forb}}}(0,r-L)
 =K_{\left.\Aop^q\right|_{\mathrm{allowed}}}(0,r-L)
      +\sum_{i=1}^{\ell}K_{\mathcal C_i}(0,r-L).
\]
The restriction imposes all hyperplane constraints on the original
integration domains and evaluation terms. Their coefficients are
nonnegative, as are those of the corrections by
\cref{lem:upper-operator,lem:lower-operator};
external $\Aop$ factors preserve nonnegativity.

The corrections are added without further restrictions on their
domains or reconstruction of internal block coordinates.
Every term reaching the noninteger jump $r-L$ has a continuous step,
since all evaluations, including corrections, have integer jumps.
Its kernel is therefore divisible by $t$, and division preserves
nonnegative coefficients. \cref{lem:rooted} proves the assertion.

For $L=0,1$, set $w_\varepsilon(u)=(1-\varepsilon)w(u)+\varepsilon/2$
with $0<\varepsilon<1$. These finite piecewise constant functions
have values in $[0,1]$ and integrals in $(0,1)$, and for every integer
$j\geq0$,
$\int_0^1u^jw_\varepsilon(u)\,du\to\int_0^1u^jw(u)\,du$
as $\varepsilon\downarrow0$.
Keep $\alpha,\beta$ fixed. Each coefficient of $\Gamma$ is polynomial
in finitely many moments and the endpoints, so
$\Gamma_{w_\varepsilon,\alpha,\beta}(G_e)$ converges coefficientwise
to $\Gamma_{w,\alpha,\beta}(G_e)$, preserving nonnegativity.
\end{proof}

\section{Proof of the main theorem}\label{sec:main-proof}

\subsection{Root summation}

\begin{proposition}\label{prop:closing}
The Ehrhart polynomial satisfies
\[
 \ehr_M(t)=G_e(t,1)+t\int_0^1\Gamma_v(G_e)\,dv.
\]
The coefficients of $G_e(t,1)$ and of $\Gamma_v(G_e)$ are
nonnegative for every $0\leq v\leq1$.
\end{proposition}

\begin{proof}
\cref{lem:rooted} gives
\[
 F_e(a,t)=\sum_{s=0}^{q-1}h_s(t)\binom{a+s}{s},
 \qquad G_e(t,X)=\sum_{s=0}^{q-1}h_s(t)X^s.
\]
Summation over $0\leq a\leq t$ and telescoping give
\[
 \ehr_M(t)=\sum_{s=0}^{q-1}h_s(t)\binom{t+s+1}{s+1}
\]
for every positive integer $t$, hence as a polynomial identity.
\cref{lem:root-summation} gives \eqref{eq:root-goal}.
Each coefficient of $\Gamma_v(G_e)$ is polynomial in $v$;
coefficient extraction commutes with its Riemann integral because
the sum is finite.

\cref{prop:kernel-positive} proves both nonnegativity
assertions: use $w=\mathbf1_{[1-v,1]}$, $\alpha=1$, $\beta=1-v$
for $\Gamma_v(G_e)$, as in \eqref{eq:gamma-specialization}, and
$w=0$, $\alpha=1$, $\beta=0$ for
$G_e(t,1)=\Gamma_{0,1,0}(G_e)$.
The specialization $\Gamma_{0,1,0}$ differs from $\Gamma_v$ at $v=0$:
since $\Gamma_0(X^s)=[\lambda^s](1-\lambda)^{-2}=s+1$,
\[
 \Gamma_0(G)=G(t,1)+\partial_XG(t,1),
 \qquad \Gamma_{0,1,0}(G)=G(t,1)
\]
for every polynomial $G(t,X)$.
\end{proof}

The local term counts the coordinate face $x_e=0$:
\[
 G_e(t,1)=F_e(0,t)=\ehr_{M\setminus e}(t).
\]
The first equality is the binomial expansion at $a=0$.
For the second, connectedness and $|E|>1$ imply that $e$ is not
a coloop, so the bases of $M\setminus e$ are the bases of $M$
avoiding $e$. Deleting the $e$-coordinate is therefore a lattice
isomorphism from $P(M)\cap\{x_e=0\}$ to $P(M\setminus e)$.
Their lattice-point counts agree at every positive integer dilation,
which proves the polynomial identity.

To analyze $\Gamma_v(G_e)$ for $0<v<1$, use the following
complementary profile for \eqref{eq:gamma-specialization}:
\begin{equation}\label{eq:root-profile}
 m(u)=
 \begin{cases}
  1-\{u\},&u\notin\Z,\\
  0,&u\in\Z,
 \end{cases}
\end{equation}
where $\{u\}$ is the fractional part. The substitution $x=1-u$
proves period-uniformity and gives
\[
 \int_0^{1-v}g(m(u))\,du=\int_v^1g(x)\,dx
\]
for every continuous $g:[0,1]\to\R$. This is the required
complementary-profile identity with $L=v$; the endpoint values are
$m(0)=0$ and $m(1-v)=m(r-v)=v$.

\begin{example}[\cref{ex:four-element}, continued]
Fix $0<v<1$ and use the profile in \eqref{eq:root-profile}.
The total jump is $2-v$, and the allowed original paths satisfy
$d_1+d_2\leq1$. Group their contributions by the number $h$
of continuous steps.

If $h=1$, there is one unit jump, one zero jump, and one continuous
jump of length $1-v$. A zero jump at the start or immediately
after an initial unit jump has weight zero. Placing both the unit
and continuous jumps in the first two coordinates violates the
constraint. The only remaining path is
\[
 (d_1,d_2,d_3)=(1-v,0,1),
\]
whose frozen weights have product $v(1-v)$.

If $h=2$ and the frozen jump is zero, placing it first gives weight
zero, while placing it third violates the constraint.
Placing it second contributes
\[
 \int_{1-v}^{1}(1-x)\,dx=\frac{v^2}{2}.
\]
If the frozen jump is one, it must be in the third coordinate.
The first two jumps sum to $1-v$, and the unit-jump weight is
$1-v$, giving the contribution $(1-v)^2$.

If $h=3$, set $s=d_1+d_2$. The final-jump equation
$d_3=2-v-s$ and the constraint require $1-v<s\leq1$.
For each $s$, the choices $0<d_1<s$ have length $s$, so the
contribution is
\[
 \int_{1-v}^{1}s\,ds=v-\frac{v^2}{2}.
\]
Since $h$ continuous steps give degree $h-1$ after division by $t$,
the allowed original kernel is
\[
 \frac{K_{\left.\Aop^3\right|_{d_1+d_2\leq1}}(0,2-v)}t
 =v(1-v)
  +\left((1-v)^2+\frac{v^2}{2}\right)t
  +\left(v-\frac{v^2}{2}\right)t^2.
\]
The normalized kernel contribution of the inserted boundary correction is
$Q_{2,1}(t,m(0))=t+1$, since the third coordinate makes the
continuous jump $1-v$. Hence
\[
 \Gamma_v(G_e)
 =(1+v-v^2)
  +\left(2-2v+\frac32v^2\right)t
  +\left(v-\frac12v^2\right)t^2.
\]
This polynomial identity extends to $v=0,1$ by continuity;
the path calculation used only noninteger total jumps.
Finally, $G_e(t,1)=t+1$, and \cref{prop:closing} gives
\[
 \ehr_M(t)=t+1+t\int_0^1\Gamma_v(G_e)\,dv
          =1+\frac{13}{6}t+\frac32t^2+\frac13t^3.
\]
The inserted boundary correction contributes $t+t^2$ to the
Ehrhart polynomial.
\end{example}

\subsection{Positivity in every degree}

An original path with $h$ continuous steps contributes degree $h-1$
after division of its kernel by $t$. We prove positivity in every
degree by finding allowed paths with each number of continuous steps
and positive frozen weights. For \eqref{eq:root-profile}, $m(0)=0$
kills an initial frozen zero step. We keep the first nonroot
coordinate $f=1$ continuous and use connected coordinate faces to
avoid integer prefixes before subsequent frozen steps.

For a basis $B$ and an element $f\notin B$, the
\emph{fundamental circuit of $f$ with respect to $B$} is the
unique circuit contained in $B\cup\{f\}$
\cite[Corollary~1.2.6]{Oxley2011}.
The fundamental bipartite graph has vertex set $E$ and
bipartition $B,E\setminus B$; an edge joins $b\in B$ to
$f\notin B$ when $B\setminus\{b\}\cup\{f\}$ is a basis.
Equivalently, $b$ belongs to the fundamental circuit of $f$.
This graph is connected if and only if the matroid is connected
\cite[Proposition~4.3.2]{Oxley2011}. Any two distinct elements of a
connected matroid belong to a common circuit
\cite[Proposition~4.1.3]{Oxley2011}.

\begin{lemma}\label{lem:faces}
Let $M$ be connected on $\{e,1,\ldots,q\}$, where $q\geq1$,
and let $f=1$.
There are nested sets
\[
 \{e,f\}=S_1\subset S_2\subset\cdots\subset S_q=E,
 \qquad |S_h|=h+1,
\]
and, for each $h=1,\ldots,q$, a coordinate face of $P(M)$ whose
unfixed coordinates are $S_h$ and form a connected matroid base
polytope. For every $0<v<1$,
the slice $x_e=v$ of each such face has a nonempty relatively open
part with all other unfixed coordinates in $(0,1)$, of dimension $h-1$.
\end{lemma}

\begin{proof}
\emph{Coordinate faces and connected minors.}
Choose a circuit containing $e,f$ and extend the independent set
obtained by removing $f$ to a basis $B$. This basis avoids $f$,
and the chosen circuit is the unique circuit in $B\cup\{f\}$.
Thus $ef$ is an edge of the connected fundamental graph.
Starting with this edge and successively adding adjacent vertices
gives the sets $S_h$ with connected induced subgraphs.

The coordinate face
\[
 x_b=1\quad(b\in B\setminus S_h),\qquad
 x_g=0\quad(g\in(E\setminus B)\setminus S_h)
\]
contains $\mathbf1_B$. Removing its fixed coordinates identifies
it with $P(N_h)$, where $N_h=(M/(B\setminus S_h))|S_h$.
For $b\in B\cap S_h$ and $g\in S_h\setminus B$,
\begin{align*}
(B\cap S_h)\setminus\{b\}\cup\{g\}\text{ is a basis of }N_h\quad \Longleftrightarrow\quad B\setminus\{b\}\cup\{g\}\text{ is a basis of }M.
\end{align*}
Indeed, the contraction basis correspondence
\cite[Corollary~3.1.8]{Oxley2011} gives the equivalence before
deletion. The basis $B\cap S_h$ survives deletion, preserving
rank, and the exchanged set contains no deleted element.
The fundamental graph of $N_h$ is therefore the connected induced
graph on $S_h$. Hence $N_h$ is connected, $\dim P(N_h)=h$,
and its affine hull is $x(S_h)=r(N_h)$.

\smallskip\noindent
\emph{Root slices.}
Let $\bar x$ be the average of all basis indicators of $N_h$.
A supporting functional maximal at $\bar x$ is maximal at every
basis indicator, hence constant on $P(N_h)$; thus $\bar x$ is
in its relative interior. Since $N_h$ is connected and has at least two elements, it has
neither loops nor coloops; every coordinate of $\bar x$ lies in $(0,1)$.
Set $v_0=\bar x_e$ and choose bases $D_0,D_1$ with
$e\notin D_0$ and $e\in D_1$. For $0<v<1$, set
\[
 x(v)=
 \begin{cases}
 \dfrac{v}{v_0}\bar x+
       \left(1-\dfrac{v}{v_0}\right)\mathbf1_{D_0},
                                      &0<v<v_0,\\[3pt]
 \bar x,                             &v=v_0,\\[3pt]
 \dfrac{1-v}{1-v_0}\bar x+
       \dfrac{v-v_0}{1-v_0}\mathbf1_{D_1},
                                      &v_0<v<1.
 \end{cases}
\]
Then $x_e(v)=v$. The positive coefficient of $\bar x$ keeps
$x(v)$ in the relative interior with all coordinates in $(0,1)$.
Since $x_e$ is nonconstant on the $h$-dimensional affine hull,
its level set has dimension $h-1$. A relative neighborhood of
$x(v)$ in that level set remains in the polytope and in $(0,1)^{S_h}$.
Restoring the fixed coordinates gives the required root slice;
for $h=1$ it is a point.
\end{proof}

\begin{proposition}\label{prop:all-degrees}
Let $M$ be the connected paving positroid in \cref{lem:rooted}. For every $0<v<1$, each coefficient of
$\Gamma_v(G_e)$ in degrees $0,\ldots,q-1$ is positive.
\end{proposition}

\begin{proof}
\emph{Allowed paths and integer prefixes.}
Use \eqref{eq:root-profile} and fix $1\leq h\leq q$. \cref{lem:faces} gives a coordinate face with $h$ unfixed
nonroot coordinates, including $f=1$. Write their jumps in order
as $y_1,\ldots,y_h$, and let $j$ count the frozen unit jumps.
Their root slice satisfies
\[
 \sum_{k=1}^{h}y_k=r-v-j,\qquad
 0<y_k<1\quad(1\leq k\leq h).
\]
Choose the integration term at these coordinates and the zero or
unit evaluation at every fixed coordinate. All points of this
slice satisfy the hyperplane inequalities, so this path term is
allowed there. Connectedness of the minor makes the displayed
sum its only affine equation.

The first step is continuous. Every frozen step therefore starts
at an integer plus a nonempty prefix $\sum_{k=1}^{a}y_k$.
For $a=h$ this is $r-v-j\notin\Z$. For $1\leq a<h$,
increasing a prefix coordinate and decreasing one outside it shows
that the prefix is nonconstant on the fixed-sum affine hull.
Since it lies in $(0,a)$, the conditions that it be integral form
finitely many proper affine hyperplanes. For $h=1$ there are
no proper nonempty prefixes.

\smallskip\noindent
\emph{A positive integral for $h\geq2$.}
These hyperplanes cannot cover the relatively open slice: a line
with direction parallel to none of them meets each at most once,
so a short segment in the slice contains a point outside them all.
At this point every frozen step starts at a noninteger, and both
weights $m(u)$ and $1-m(u)$ are positive.

Eliminate the last continuous coordinate by
$y_h=r-v-j-\sum_{k=1}^{h-1}y_k$. Its coefficient in the total-jump
equation is $1$, so changing to the final point has derivative $1$.
Let $\Omega\subseteq(0,1)^{h-1}$ be the allowed domain of the
remaining integration variables and $\omega$ the product of the
frozen weights after elimination. Projection onto these variables
is an affine bijection from the fixed-sum affine hull to
$\R^{h-1}$, so the relatively open root slice projects to an open
subset of $\Omega$. Only these continuous coordinates are integrated;
the evaluation factors fix the remaining coordinates.

A sufficiently small closed box
$Q=\prod_{k=1}^{h-1}[a_k,b_k]$ about the chosen point lies in
$\Omega$, has $a_k<b_k$, and avoids the exceptional hyperplanes.
The frozen weights are continuous and positive on $Q$, so
$\omega\geq\eta>0$ there. Hence
\[
 \int_{\Omega}\omega(y_1,\ldots,y_{h-1})
                         \,dy_1\cdots dy_{h-1}
 \geq\eta\prod_{k=1}^{h-1}(b_k-a_k)>0.
\]
With no frozen steps the product is $1$. Thus the selected term
has a positive coefficient of $t^h$ in the allowed kernel.

\smallskip\noindent
\emph{One continuous step and the coefficients.}
For $h=1$, the continuous step is $f=1$ and
$y_1=r-v-j\in(0,1)$ forces $j=r-1$, $y_1=1-v$.
Every later starting point has fractional part $1-v$.
The $q-r$ frozen zero jumps have weight $v$ and the $r-1$
frozen unit jumps have weight $1-v$, giving the positive
coefficient $v^{q-r}(1-v)^{r-1}$ of $t$.

For each $h$, division by $t$ gives a positive contribution in
degree $h-1$. All other allowed terms and boundary corrections
have nonnegative coefficients by \cref{prop:kernel-positive}. Therefore
\[
 [t^{h-1}]\Gamma_v(G_e)>0\qquad(1\leq h\leq q).
\]
This completes the proof.
\end{proof}

\begin{proof}[Proof of \cref{thm:main}]
The empty-ground-set and singleton cases have point base polytopes
and Ehrhart polynomial $1$. Suppose $M$ is connected with at least
two elements. In rank one or corank one it is uniform, so
\cite[Theorem~1.2]{Ferroni2021} applies. Otherwise
$2\leq r\leq n-2$ and $q=n-1$.
For $1\leq j\leq q$, \cref{prop:all-degrees} makes
$[t^{j-1}]\Gamma_v(G_e)$ positive on $(0,1)$.
This coefficient is polynomial in $v$ and hence continuous;
it has a positive lower bound on some closed interval in $(0,1)$,
so its integral is positive. \cref{prop:closing},
including nonnegativity of the local term, gives
$[t^j]\ehr_M(t)>0$. Ehrhart's theorem gives $\ehr_M(0)=1$,
and connectedness gives $\dim P(M)=q$.

For general $M$, its connected components $M_1,\ldots,M_c$ are
positroids in the induced orders, being restrictions
\cite[Proposition~3.5]{ARW2016}. Every circuit of $M_i$ is a
circuit of $M$, of size at least $r(M)\geq r(M_i)$; hence each
$M_i$ is paving. Their Cartesian product decomposition gives
\[
 \ehr_M(t)=\prod_{i=1}^{c}\ehr_{M_i}(t),
 \qquad \dim P(M)=\sum_{i=1}^{c}\dim P(M_i).
\]
Each factor has positive coefficients throughout its degree range.
Every integer between zero and the sum of these degrees is a sum
of integers from the respective ranges. Its coefficient in the
product therefore contains a positive product, and all remaining
summands are nonnegative.
\end{proof}

\begin{remark}
The proof requires a common cyclic order. The open caps are disjoint
for every paving matroid, but its dependent hyperplanes need not be
blocks in one order. Such an order is needed to interpret all the
single-block identities as restrictions of the same path integration
domains.
\end{remark}

%\noindent\textbf{Ethical approval:} Not applicable.

%\noindent\textbf{Conflict interest:} None.

%\noindent\textbf{Declaration of competing interest:} The authors declared that they had no conflicts of interest with respect to %their authorship or the publication of this article.

%\noindent\textbf{Data availability:} Data availability is not applicable to this article as no new data were created or analyzed in %this study.

%\noindent\textbf{Funding:} This work was partially supported by the National Natural Science Foundation of China***********.

\vspace{5mm}
\noindent
{\small \textbf{Acknowledgments:}}
%The authors would like to thank the anonymous referees for their valuable suggestions for improving the presentation.
Feihu Liu was partially supported by the Postdoctoral Fellowship Program and China Postdoctoral Science Foundation (Grant No. BX2026002).

%\section*{Declaration of AI Assistance}
\noindent{\small \textbf{Declaration of AI Assistance:}}
During the preparation of this manuscript, the author used ChatGPT to assist in exploring possible approaches, checking technical details, and improving the exposition. All mathematical arguments, computations, proofs, and results were independently verified by the author. The author take full responsibility for the content of this manuscript.

\end{document}